\documentclass[12pt]{amsart}

\newtheorem{theorem}{Theorem}[section]
\newtheorem{lemma}[theorem]{Lemma}

\theoremstyle{definition}

\theoremstyle{remark}

\begin{document}
\title[On Pseudo-Einstein Real Hypersurfaces]{On Pseudo-Einstein Real Hypersurfaces}
\author{MAYUKO KON}
\address{Faculty of Education, Shinshu University, 6-Ro, Nishinagano, Nagano City 380-8544, Japan}
\email{mayuko$\_$k@shinshu-u.ac.jp}

\maketitle
\begin{abstract}
Let $M$ be a real hypersurface of a complex space form $M^n(c)$, $c\neq0$, $n\geq 3$. We show that the Ricci tensor $S$ of $M$ satisfies $S(X,Y)=ag(X,Y)$ for any vector fields $X$ and $Y$ on the holomorphic distribution, $a$ being a constant, if and only if $M$ is a pseudo-Einstein real hypersurface.
\end{abstract}
\vspace{5mm}

\section[1]{Introduction}
 A Riemannian manifold is said to be Einstein if the Ricci tensor $S$ is a constant multiple of the metric tensor, that is, $S=\rho g$. In the theory of subspace, Fialkow \cite{Fi} classified Einstein hypersurfaces in spaces of constant curvature (see also Ryan \cite{Ry}).

On the other hand,  it is known that complex space forms with non-zero constant holomorphic sectional curvature  do not admit Einstein real hypersurfaces. However, as Kon\cite{Ko} discovered, there is a nice class of real hypersurfaces satisfying $S(X,Y) = a g(X,Y) + b\eta(X)\eta(Y)$ for all tangent vectors $X$ and $Y$, where $S$ is the Ricci tensor and $a$, $b$ are constants. Here a 1-form $\eta$ is defined by $\eta(X)=g(X,\xi)$, where $\xi$ is the structure vector field. Such real hypersurfaces are said to be pseudo-Einstein. For $n \geq 3$, the pseudo-Einstein real hypersurfaces were classified by Kon\cite{Ko}  for the complex projective space $CP^n$ and by Montiel \cite{Mo} for the complex hyperbolic space $CH^n$ (see also Cecil and Ryan \cite{CR}).  For $n = 2$, the classification problem completed by Kim and Ryan \cite{KR}, Ivey and Ryan \cite{IR}.

The purpose of this paper is to study the following condition for the Ricci tensor $S$ of a real hypersurface $M$ :
$$S(X,Y)=ag(X,Y), \hspace {1cm} X, Y \in H$$
where $a$ is constant and $H$ denotes the {\it holomorphic distribution} on $M$ defined by $H(x) =\{X\in T_x(M)|\eta(X)=0\}$.

If $S$ satisfies the pseudo-Einstein condition, then it satisfies the condition above.  An orthogonal splitting of the tangent space of $M$ is given by $T(M)= {\rm span}\{\xi\} \oplus H$.  We study the Ricci tensor $S$ with respect to the condition on $H$ and prove the following\\

\noindent{\textbf{Theorem.} \textit{Let $M$ be a real hypersurface of a complex space form $M^n(c)$, $c\neq 0$, $n\geq 3$. The Ricci tensor $S$ of $M$ satisfies $S(X,Y)=ag(X,Y)$ for any $X,Y\in H$, $a$ being a constant, if and only if $M$ is a pseudo-Einstein real hypersurface.}

\section[2]{Preliminaries} 

In this section we prepare some basic formulas for real hypersurfaces of complex space forms. For the general theory of real hypersurfaces we refer to Niebergall and Ryan \cite{NR}.

 Let $M^n(c)$ denote the complex space form of complex dimension $n$ (real dimension $2n$) with constant holomorphic sectional curvature $4c$. We denote by $J$ the almost complex structure of $M^n(c)$. The Hermitian metric of $M^n(c)$ will be denoted by $G$.

Let $M$ be a real $(2n-1)$-dimensional hypersurface immersed in $M^n(c)$.  We denote by $g$ the Riemannian metric induced on $M$ from $G$. We take the unit normal vector field $N$ of $M$ in $M^n(c)$. Then the {\it structure vector field} $\xi$ on $M$ is defined so that
$$ \xi =-JN,  \hspace{1cm} \xi \in T(M).$$
This gives an orthogonal splitting of the tangent space
$$T(M)= {\rm span}\{\xi\} \oplus H.$$
On the tangent space we define a linear operator $\phi$ :
$$\phi X = JX - g(X,\xi)\xi, \hspace {1cm} \phi : T(M) \longrightarrow T(M).$$
Second, we define a 1-form $\eta$ by $\eta(X)=g(X,\xi)$, then
\begin{equation*}\label{xx}
\begin{split}
&\eta(\phi X)=0,  \quad \phi \xi=0,\\
&g(\phi X,Y)+g(X,\phi Y) = 0,\\
&g(\phi X,\phi Y)=g(X,Y)-\eta(X)\eta(Y). 
\end{split}
\end{equation*}
Thus $(\phi,\xi,\eta,g)$ defines an almost contact metric structure on $M$.

We denote by $\tilde{\nabla}$ the operator of covariant differentiation in $M^n(c)$, and by $\nabla$ the one in $M$ determined by the induced metric. Then the {\it Gauss and Weingarten formulas} are given respectively by
\begin{equation*}
\tilde{\nabla}_XY={\nabla}_XY+g(AX,Y)N,  \hspace{1cm} \tilde{\nabla}_XN = -AX
\end{equation*}
for any vector fields $X$ and $Y$ tangent to $M$. We call $A$ the {\it shape operator} of $M$.

For the contact metric structure on $M$, we have
\begin{equation*}
{\nabla}_X\xi=\phi AX, \hspace{1cm} ({\nabla}_X\phi)Y=\eta(Y)AX-g(AX,Y)\xi.
\end{equation*}
From this we obtain
\begin{equation*}
g({\nabla}_X\phi Y,Z)=g(\phi{\nabla}_X Y,Z),
\end{equation*}
where $X\in T(M)$ and $Y, Z\in H$.

We denote by $R$ the Riemannian curvature tensor field of $M$. Then the {\it equation of Gauss} is given by
\begin{equation*}\label{xx}
\begin{split}
R(X,Y)Z&= c(g(Y,Z)X - g(X,Z)Y + g(\phi Y,Z)\phi X\\
       &\quad - g(\phi X,Z)\phi Y - 2g(\phi X,Y)\phi Z )\\
       &\quad + g(AY,Z)AX - g(AX,Z)AY,
\end{split}
\end{equation*}
and the {\it equation of Codazzi} by
\begin{equation*}
(\nabla_XA)Y-(\nabla_YA)X = c(\eta(X)\phi Y - \eta(Y)\phi X - 2g(\phi X, Y)\xi).
\end{equation*}
From the equation of Gauss, the Ricci tensor $S$ of $M$ is given by
\begin{equation}\label{1}
\begin{split}
S(X,Y)&=(2n+1)cg(X,Y)-3c\eta (X)\eta (Y) \\
      & \quad+ {\mathrm{tr}}Ag(AX,Y) -g(AX,AY),
\end{split}
\end{equation}
where ${\mathrm{tr}}A$ is the trace of $A$.

A hypersurface $M$ of a complex space form  $M^n(c)$ is called a {\it  Hopf hypersurface} if the structure vector field $\xi$ is a principal vector, that is, $A\xi = \alpha\xi$, $\alpha=g(A\xi,\xi)$. We define the subspace $L(x)\subset T_x(M)$ as the smallest subspace that contains $\xi$ and is invariant under the shape operator $A$. Then $M$ is Hopf if and only if $L(x)$ is one-dimensional at each point $x$. 

We recall the notion of {\it pseudo-Einstein real hypersurfaces}. A real hypersurface $M$ of a complex space form  $M^n(c)$ is said to be pseudo-Einstein if there are constants $a$ and $b$ such that the Ricci tensor $S$ of $M$ satisfies
$$S(X,Y)=ag(X,Y)+b\eta(X)\eta(Y)$$
for all tangent vectors $X$ and $Y$. We remark that any pseudo-Einstein real hypersurface satisfies that $S(X,\xi)=0$ for all $X\in H$. This means that the structure vector field $\xi$ is an eigenvector field of the Ricci tensor of type (1,1). Such a hypersurface was studied by Kon \cite{MaK}.
 
We consider the condition that the Ricci tensor $S$ of $M$ satisfies
$$S(X,Y) = ag(X,Y), \hspace{1cm} X, Y \in H,$$
where $a$ is a constant. If $M$ is pseudo-Einstein, then it satisfies this condition. So it is weaker than that of pseudo-Einstein.  Our condition is equivalent to that $S(\phi^2 X,\phi^2Y)=ag(\phi^2 X,\phi^2 Y)$ for any vector fields $X$ and $Y$ or equivalently
\begin{equation*}
\begin{split}
S(X,Y)&=\eta(X)S(\xi,Y) + \eta(Y)S(X,\xi) -\eta(X)\eta(Y) S(\xi,\xi)\\
         & \quad +ag(X,Y) - a\eta(X)\eta(Y).
\end{split}
\end{equation*}
So the pseudo-Einstein condition is equivalent to that $S(X,Y) = ag(X,Y)$, $S(X,\xi)=0, X, Y \in H$.

However, our result states that the condition $S(X,Y) = ag(X,Y),  X, Y \in H$ is equivalent to the pseudo-Einstein condition. Therefore, pseudo-Einstein real hypersurface are determined the condition on the holomorphic distribution $H$ on $M$.

\section[3]{A condition on the Ricci tensor}

Let $M$ be a connected real hypersurface of $M^n(c)$ $(n\geq 3, c\neq 0)$. We consider the symmetric tensor field $\phi A\phi$ of type (1,1) on $M$. As a point $x$ of $M$ we take an orthonormal basis $\{\xi, v_1,\hdots,v_{2n-2}\}$ in the tangent space  $T_{x}(M)$ at $x \in M$ such that
$$\phi A\phi \xi=0, \hspace{1cm} \phi A\phi v_i=-a_i v_i, \hspace{1cm}1 \leq i \leq 2n-2.$$
Then
$$A\phi v_i = a_i \phi v_i + \eta(A\phi v_i)\xi,  \hspace{1.0cm}  1\leq i \leq 2n-2.$$
Suppose that the Ricci tensor $S$ of $M$ satisfies a condition
$$S(X,Y)=ag(X,Y), \hspace{1cm} X, Y\in H, $$
$a$ being a constant. Then  (\ref{1}) implies
$$S(\phi v_i,\phi v_j) = - \eta(A\phi v_i)\eta(A\phi v_j) = 0$$
for $i\neq j, 1\leq i, j \leq 2n-2$. So we may assume
$$\eta(A\phi v_j) = 0,  \hspace{1.0cm}  2\leq i \leq 2n-2.  $$

We notice that ${\rm dim}L(x) \leq 2$. If  ${\rm dim}L(x) = 2$, then $L(x)$ is spanned by $\xi, A\xi$.

Taking a new orthonormal basis $\{ \xi, e_1 = \phi v_1,\hdots,e_{2n-2}=\phi v_{2n-2} \} $, we obtain
$$A\xi = \alpha\xi + h_1 e_1, \hspace{1.0cm} Ae_1=a_1 e_1 + h_1 \xi,$$
$$Ae_i = a_i e_i, \hspace{1.0cm} 2\leq i \leq 2n-2,$$
where $h_1 = \eta(Ae_1)$. Then, from the assumption on the Ricci tensor $S$, we have
$$(2n+1)c + ({\rm tr}A)a_1 -(a_1^2+h_1^2)=a,$$
$$(2n+1)c + ({\rm tr}A)a_i -a_i^2=a, \hspace{1.0cm} 2\leq i \leq 2n-2.  $$
Each of $a_2, \hdots, a_{2n-2}$ is a root of the quadratic equation
\begin{equation}\label{3}
X^2 -({\rm tr}A)X + a- (2n+1)c = 0.
\end{equation}
Thus at most two $a_{i}^{, }s,  2\leq i \leq 2n-2,$ can be distinct at each point. Let us denote them by $\beta$ and $\gamma$.

The argument above applies to each point $x$ of $M$. Since ${\rm tr}A$ is differentiable, it follows that roots of  the quadratic equation are differentiable functions.

If $M$ is Hopf, then $M$ is a pseudo-Einstein. Therefore, in the following, we assume that $M$ is not Hopf. We work in an open set where $A\xi -\alpha\xi$ does not vanish, that is, $A\xi = \alpha\xi + h_1 e_1$, $h_1$ being a nonvanishing function and $e_1$ is a unit vector field orthogonal to $\xi$, $\eta(e_1)=0$. We notice that $\alpha=g(A\xi,\xi)$ and $a_1=g(Ae_1,\xi)=\eta(Ae_1)$ are differentiable.

Let us restrict ourselves to a neighborhood of a point $x$ where $\beta\neq \gamma$. We assume that $\beta$ appears $p$ times and $\gamma$ appears $2n-3-p$ times. By the quadratic equation above, we obtain

\begin{equation}\label{3}\begin{split}
& {\mathrm{tr}}A=\beta+\gamma,\\
& a-(2n+1)c=\beta\gamma.
\end{split}\end{equation}
Since we have 
$${\rm tr}A=\alpha + a_1 + p\beta + (2n-3-p)\gamma,$$
it follows
$$\alpha + a_1 + (2n-5)\gamma + (p-1)(\beta - \gamma) = 0.$$
From this we see that $p$ is a constant. We also have
\begin{equation}\label{4}
h_1^2=({\mathrm{tr}}A)(a_1-\beta)-(a_1^2-\beta^2)
\end{equation}
and, by (\ref{3}),
\begin{equation}\label{5}
h_1^2=(a_1-\beta)(\gamma-a_1).
\end{equation}

We define two distributions $T_{\beta}$ and $T_{\gamma}$ as follows:
$$ T_{\beta}(x)=\{X\in T_{x}(M)\mid AX=\beta X \}$$
$$ T_{\gamma}(x)=\{X\in T_{x}(M)\mid AX=\gamma X \}.$$
We take a local orthonormal basis $\{\xi, e_1, e_2, \hdots, e_{2n-2} \}$ such that
$$A\xi = \alpha\xi + h_1 e_1, \hspace{1cm}Ae_1 = a_1 e_1+ h_1 \xi,$$
$$ Ae_i = \beta e_i, \hspace{1cm} 2\leq i \leq p+1, $$
$$Ae_j = \gamma e_j, \hspace{1cm} p+2\leq j \leq 2n-2.$$ 
Then, $\{e_2, \hdots, e_{p +1} \}$ is an orthonormal basis for $T_{\beta}$ and  $\{e_{p+2}, \hdots, e_{2n-2} \}$ is an orthonormal basis for $T_{\gamma}$. We see
$$T_x (M) = L(x)\oplus T_{\beta}(x) \oplus T_{\gamma}(x)$$
at each point $x$ of $M$.

Here, using the equation of Codazzi, we prepare some basic formulas:

\begin{lemma}
With respect to a local orthonormal basis $\{\xi, e_1,\cdots,e_{2n-2}\}$, we have
\begin{eqnarray}
& &(a_j-a_k)g(\nabla_{e_i}e_j,e_k)-(a_i-a_k)g(\nabla_{e_j}{e_i},e_k)=0,\label{6}\\
& &(a_j-a_1)g(\nabla_{e_i}e_j,e_1)-(a_i-a_1)g(\nabla_{e_j}e_i,e_1)\label{7}\\
& &\quad +h_1(a_i+a_j)g(\phi e_i,e_j)=0,\nonumber\\
& &\{2c-2a_ia_j+\alpha(a_i+a_j)\}g(\phi e_i,e_j) - h_1 g(\nabla_{e_i}e_j,e_1) \label{8}\\
& &\quad+ h_1g(\nabla_{e_j}e_i,e_1)=0,\nonumber\\
& &(a_j-a_i)g(\nabla_{e_i}e_j, e_i)- (e_ja_i)=0,\label{9}\\
& &(a_1-a_i)g(\nabla_{e_i}e_1,e_i)-(e_1a_i)=0,\label{10}\\
& &(a_1-a_j)g(\nabla_{e_i}e_1,e_j)+(a_j-a_i)g(\nabla_{e_1}e_i,e_j)\label{11}\\
& &\quad+ a_ih_1 g(\phi e_i,e_j)=0,\nonumber\\
& &(2c-2a_1a_i + \alpha(a_i+a_1))g(\phi e_i,e_1) \label{12}\\
& &\quad+ h_1 g(\nabla_{e_1}e_i,e_1)+(e_ih_1)=0,\nonumber\\
& &h_1(2a_i+a_1)g(\phi e_i, e_1) + (a_1-a_i)g(\nabla_{e_1}e_i, e_1) \label{13}\\
& &\quad+ (e_ia_1)=0,\nonumber\\
& &h_1 g(\nabla_{e_i}e_1,e_i)- (\xi a_i)=0,\label{14}\\
& &(c + a_i \alpha - a_ia_j) g(\phi e_i, e_j) + h_1 g(\nabla_{e_i}e_1, e_j)\label{15}\\
& & \quad+ (a_j-a_i) g(\nabla_{\xi}e_i,e_j)=0,\nonumber\\
& &(c +a_i\alpha - a_ia_1 + h_1^2)g(\phi e_i,e_1) + (a_1-a_i)g(\nabla_{\xi}e_i,e_1)\label{16}\\
& &\quad +(e_ih_1)=0,\nonumber\\
& &h_1 (\alpha - 3a_i) g( e_1,\phi e_i) + h_1 g(\nabla_{\xi}e_i, e_1)\label{17}\\
& &\quad + (e_i\alpha) =0,\nonumber\\
& &(e_1h_1)- (\xi a_1)=0,\label{28}\\
& &(e_1\alpha) - (\xi h_1)=0,\label{19}\\
& &(c+a_1\alpha - a_1a_i - h_1^2) g(\phi e_1,e_i) - (a_1-a_i) g(\nabla_{\xi}e_1, e_i) \label{20}\\
& &\quad+ h_1 g(\nabla_{e_1}e_1, e_i)=0,\nonumber
\end{eqnarray}
where $i\neq j$, $j\neq k$, $k\neq i$, $i,j,k\geq 2$, and $a_i, a_j, a_k$  are $\beta$  or $\gamma$.
\end{lemma}

\begin{lemma}
Let $M$ be a non-Hopf real hypersurface of $M^n(c)$, $n\geq 3$, $c\neq 0$. Suppose that the Ricci tensor $S$ satisfies $S(X,Y)=ag(X,Y)$ for any $X, Y\in H$, $a$ being a constant. If $\beta\neq \gamma$, then the orthonormal basis $\{\xi, e_1,\cdots,e_{2n-2},\}$ satisfies that
$g(\phi e_x, e_y)=0$  for any $e_x,\ e_y\in T_\beta$ and $g(\phi e_s,e_t)=0$ for any $e_s,\ e_t\in T_{\gamma}.$
\end{lemma}

\begin{proof}
By (\ref{11}) and (\ref{15}), we obtain
\begin{eqnarray*}
& &(a_1-\beta)g(\nabla_{e_x}e_1,e_y) + \beta h_1 g(\phi e_x, e_y)=0,\\
& &h_1 g(\nabla_{e_x}e_1, e_y) + (c+\alpha \beta - \beta^2) g(\phi e_x, e_y)=0.
\end{eqnarray*}
From these equations and (\ref{5}), we have
\begin{eqnarray*}
(a_1-\beta) \{\beta\gamma - \beta a_1-c-\alpha\beta + \beta^2\} g(\phi e_x,e_y)=0.
\end{eqnarray*}
Thus, if there exist $e_x$ and $e_y$ that satisfy $g(\phi e_x,e_y)\neq 0$, then we have
\begin{equation}\label{21}
c=\beta\{\beta +\gamma - (a_1+\alpha)\}.
\end{equation}
Similarly, if there exist $e_s$ and $e_t$ such that $g(\phi e_s, e_t)\neq 0$, then we obtain
\begin{equation*}
c=\gamma\{\gamma+\beta - (a_1+\alpha)\}.
\end{equation*}
Therefore the assumption $\beta \neq \gamma$ yields $\beta + \gamma=a_1+\alpha$, and hence $c=0$. This is a contradiction. So we have two cases:\\
\begin{itemize}
\item[(I)] We have $g(\phi e_x, e_y)=0$  for any $e_x,\ e_y\in T_\beta$ and $g(\phi e_s,e_t)=0$ for any $e_x,\ e_y\in T_{\gamma}.$
\item[(II)] There exist $e_x, e_y \in T_{\beta}$ such that $g(\phi e_x, e_y) \neq 0$, and for any $e_s, e_t \in T_{\gamma}$,  $g(\phi e_t, e_s) = 0$, or there exist $e_t, e_s \in T_{\gamma}$ such that $g(\phi e_t, e_s) \neq 0$, and for any $e_x, e_y \in T_{\gamma}$,  $g(\phi e_x, e_y) = 0$.\\
\end{itemize}

We shall show that the case (II) does not occur. To this purpose, it is sufficient to consider the case that there exist $e_x, e_y \in T_{\beta}$ such that $g(\phi e_x, e_y) \neq 0$, and  $g(\phi e_t, e_s) = 0$  for any $e_s, e_t \in T_{\gamma}$.\\

In the following we put $\phi e_1=\mu e_2 + \nu e_{p+2}$, $e_2\in T_\beta, e_{p+2}\in T_\gamma$,  by taking a suitable orthonormal basis.

There exist $e_x$ and $e_s$ such that $g(\phi e_x, e_s)\neq 0$. First we show that $\beta$ and $\gamma$ are constant. By (\ref{3}), 
\begin{equation*}
a=(2n+1)c+\beta {\rm{tr}}A  -\beta^2 = (2n+1)c + \beta\gamma.
\end{equation*}
Thus $\beta\gamma$ is constant.  On the other hand, since there exist $e_x$ and $e_y$ such that $g(\phi e_x, e_y)\neq 0$,  (\ref{3}) and  (\ref{21}) imply
\begin{eqnarray*}
c&=& \beta\{{\rm{tr}}A - (a_1+\alpha)\} \\
 &=& p\beta^2 + (2n-3-p)\beta\gamma.
\end{eqnarray*}
Since $\beta\gamma$ is constant, we see that $\beta$ is constant, and hence $\gamma$ is also constant.

We compute the right hand side of
\begin{eqnarray*}
& &g(R(e_t,e_y)e_y,e_t) \\
& &= g(\nabla_{e_t}\nabla_{e_y}e_y, e_t) - g(\nabla_{e_y}\nabla_{e_t} e_y, e_t) - g(\nabla_{[e_t,e_y]} e_y, e_t).
\end{eqnarray*}
for any $e_y\in T_{\beta}$ and $e_t\in T_{\gamma}$. Using (\ref{9}) and (\ref{10}), we have
\begin{equation}\label{22}\begin{split}
&g(\nabla_{e_x}e_x,e_s)=0,\quad g(\nabla_{e_s}e_s,e_x)=0,\\
&g(\nabla_{e_x}e_x, e_1)=0,\quad g(\nabla_{e_s}e_s, e_1)=0\
\end{split}\end{equation}
for any $e_x\in T_\beta$ and $e_s\in T_\gamma$.  Using these equations and $g(\nabla_{e_x}e_x,\xi)=g(\nabla_{e_s}e_s,\xi)=0$, we have $g(\nabla_{e_t}\nabla_{e_y}e_y, e_t)=-g(\nabla_{e_y}e_y,\nabla_{e_t}e_t)=0$. 

On the other hand,  for our orthonormal basis $\{\xi, e_1, e_x, e_s\}$, we compute
\begin{eqnarray*}
& &g(\nabla_{e_y}\nabla_{e_t}e_y,e_t)\\
& & = -g(\nabla_{e_t}e_y, \nabla_{e_y}e_t)\\
& &=-g(\nabla_{e_t}e_y,\xi)g(\xi, \nabla_{e_y}e_t) - g(\nabla_{e_t}e_y, e_1)g(e_1,\nabla_{e_y}e_t)\\
& &\quad -\sum_{x} g(\nabla_{e_t}e_y, e_x)g(e_x, \nabla_{e_y}e_t) -\sum_s g(\nabla_{e_t}e_y, e_s)g(e_s, \nabla_{e_y}e_t).
\end{eqnarray*}
When $y=z$, we have $g(e_z, \nabla_{e_y}e_t)=0$ by (\ref{22}). When $y\neq z$, by (\ref{6}),
\begin{equation*}
(\gamma-\beta)g(\nabla_{e_y}e_t, e_z)=0.
\end{equation*}
Since $\beta\neq \gamma$, we obtain $g(\nabla_{e_y}e_t, e_z)=0$. Hence we have 
\begin{equation*}
\sum_x g(\nabla_{e_t}e_y, e_x)g(e_x, \nabla_{e_y}e_t)=0.
\end{equation*}
 Similar computation using (\ref{6}) and (\ref{22}) gives $g(\nabla_{e_t}e_y, e_u)=0$. So we have
\begin{equation*}
g(\nabla_{e_y}\nabla_{e_t} e_y, e_t) = \beta\gamma g(\phi e_t, e_y)^2 - g(\nabla_{e_t}e_y, e_1)g(e_1, \nabla_{e_y}e_t).
\end{equation*}
Next, we compute
\begin{eqnarray*}
& &g(\nabla_{[e_t,e_y]}e_y,e_t)\\
& &=g(\nabla_{\xi}e_y, e_t)g(\xi, [e_t,e_y])+ g(\nabla_{e_1}e_y, e_t)g(e_1, [e_t, e_y])\\
& &\quad +\sum_{z} g(\nabla_{e_z}e_y,e_t)g(e_z, [e_t,e_y])+ \sum_{u}g(\nabla_{e_u}e_y,e_t)g(e_u, [e_t,e_y]).
\end{eqnarray*}
Using (\ref{6}) and (\ref{22}), we have 
\begin{equation*}
g(\nabla_{e_u}e_y, e_t)=0,\quad g(\nabla_{e_z}e_y,e_t)=-g(e_y,\nabla_{e_z}e_t)=0.
\end{equation*}
So we obtain
\begin{eqnarray*}
& &g(\nabla_{[e_t,e_y]}e_y,e_t)\\
& &= -(\beta+\gamma)g(\phi e_t, e_y)g(\nabla_{\xi}e_y, e_t)\\
& &\quad +g(\nabla_{e_1}e_y,e_t)g(e_1,\nabla_{e_t}e_y) - g(\nabla_{e_1}e_y,e_t)g(e_1,\nabla_{e_y}e_t).
\end{eqnarray*}
Summarizing the above we have
\begin{eqnarray}\label{23}
& &g(R(e_t,e_y)e_y,e_t)\nonumber\\
& &=-\beta\gamma g(\phi e_t,e_y)^2 + g(\nabla_{e_t}e_y, e_1)g(e_1,\nabla_{e_y}e_t)\\
& &\quad +(\beta+\gamma)g(\phi e_t, e_y)g(\nabla_{\xi}e_y,e_t)\nonumber \\
& &\quad -g(\nabla_{e_1}e_y, e_t)g(e_1,\nabla_{e_t}e_y) + g(\nabla_{e_1}e_y,e_t)g(e_1,\nabla_{e_y}e_t).\nonumber
\end{eqnarray}
By (\ref{8}), (\ref{11}), (\ref{15}) and $\beta\neq \gamma$, we obtain
\begin{eqnarray*}
& &g(\nabla_{e_y}e_t, e_1)-g(\nabla_{e_t}e_y. e_1) =\frac{1}{h_1}\{2c-2\beta\gamma + \alpha(\beta+\gamma)\} g(\phi e_y, e_t),\\
& &g(\nabla_{e_1}e_y, e_t)=\frac{a_1-\gamma}{\beta-\gamma}g(\nabla_{e_y}e_1, e_t) + \frac{\beta h_1}{\beta-\gamma}g(\phi e_y,e_t),\\
& &g(\nabla_{\xi}e_y, e_t)=\frac{1}{\beta-\gamma}(c+\beta\alpha-\beta\gamma)g(\phi e_y, e_t) + \frac{h_1}{\beta-\gamma}g(\nabla_{e_y}e_1,e_t).
\end{eqnarray*}
Moreover, (\ref{7}) and (\ref{8}) imply that
\begin{eqnarray*}
& &g(\nabla_{e_y}e_t, e_1)\\
& & =\frac{-1}{h_1(\gamma-\beta)} \{h_1^2(\beta+\gamma) + (\beta-a_1)(2c-2\beta\gamma + \alpha(\beta+\gamma)\}g(\phi e_y,e_t),\\
& &g(\nabla_{e_t}e_y,e_1)\\
& &=\frac{-1}{h_1(\gamma - \beta)} \{h_1^2(\beta+\gamma) + (\gamma-a_1)(2c-2\beta\gamma+\alpha(\beta+\gamma)\}g(\phi e_y, e_t).
\end{eqnarray*}
Substituting these equations into (\ref{23}), and using  (\ref{5}) and (\ref{21}), we have
\begin{eqnarray*}
& &g(R(e_t,e_y)e_y,e_t)\\
& &=\{(c-2\beta\gamma)-2(\alpha+a_1-\beta)(a_1+\alpha - 2\beta-\gamma)\}g(\phi e_y,e_t)^2.
\end{eqnarray*}
On the other hand, by the equation of Gauss,
\begin{equation*}
g(R(e_t,e_y)e_y, e_t)=c+3cg(\phi e_y,e_t)^2 + \beta\gamma.
\end{equation*}
From these it follows that
\begin{equation*}
c+\beta\gamma = \{-2(c+\beta\gamma)-2(\alpha+a_1-\beta)(\alpha+a_1-2\beta-\gamma)\}g(\phi e_y,e_t)^2.
\end{equation*}
for any $e_y \in T_\beta, e_t \in T_\gamma$. Since $g(\phi e_s, e_t)=0$ for any $s$ and $t$ and $\phi e_1=\mu_2e_2 + \mu_{p+1}e_{p+2}$, we see that $g(\phi e_2, e_{p+2})=0$.  So we have
\begin{equation}\label{24}
(a_1+\alpha-\beta)(a_1+\alpha -2\beta-\gamma)=0.
\end{equation}
 Combining these equations with (\ref{21}), we have
\begin{equation*}
\beta(\beta + 2\gamma - a_1 -\alpha)=0.
\end{equation*}
Since $c=-\beta\gamma$, we have $\beta\neq 0$. From which it follows
\begin{equation}\label{25}
\beta+2\gamma = a_1+\alpha.
\end{equation}
From (\ref{24}), we have $a_1+\alpha=\beta$ or $a_1+\alpha = 2\beta+\gamma$. When $a_1+\alpha=\beta$, by (\ref{25}), we see that $\gamma=0$. This is a contradiction. So we have $a_1+\alpha=2\beta+\gamma$. Then (\ref{25}) implies that $\beta=\gamma$. Again, this is a contradiction.

\end{proof}

\begin{lemma}
Let $M$ be a non-Hopf real hypersurface of $M^n(c)$, $n\geq 3$, $c\neq 0$. Suppose that the Ricci tensor $S$ satisfies $S(X,Y)=ag(X,Y)$ for any $X, Y\in H$, $a$ being a constant.  If $\beta\neq \gamma$, then $\phi e_1\in T_\beta$ or $\phi e_1 \in T_\gamma$.
\end{lemma}

\begin{proof}
As a result of Lemma 3.2, we have $g(\phi e_x, e_y)=0$  for any $e_x,\ e_y\in T_\beta$ and $g(\phi e_s,e_t)=0$ for any $e_x,\ e_y\in T_{\gamma}.$ We can put $\phi e_1=\mu e_2+ \nu e_{p+2}$ by taking a suitable orthonormal basis of $T_\beta$ and $T_\gamma$. Then we have
\begin{equation*}
\phi^2e_1=-e_1=\mu\phi e_2 + \nu \phi e_{p+2}.
\end{equation*}
Since $\phi e_2 \in {\rm span}\{e_1, e_{p+2}, \cdots, e_{2n-2} \}$, if $\mu\neq 0$, then we see that $e_1+ \nu \phi e_{p+2} \in  {\rm span}\{e_1, e_{p+2}, \cdots, e_{2n-2} \}$. By Lemma 3.2, we have $\nu \phi e_{p+2} \in {\rm span} \{e_1\} $. If $\nu\neq 0$, then we may put $\phi e_{p+2}=e_1$. This contradicts to the assumption that $\mu\neq 0$. Thus we see that if $\mu\neq 0$, then $\nu=0$ and then we can take $\phi e_1=e_2$. On the other hand, if $\mu=0$, then $\phi e_1= e_{p+2}$.
\end{proof}

It is sufficient to consider the case that $\phi e_1\in T_\beta$. In the following, we put $\phi e_1=e_2$.

\begin{lemma}
Let $M$ be a non-Hopf real hypersurface of $M^n(c)$ $(n\geq 3, c\neq 0)$. Suppose that the Ricci tensor $S$ satisfies $S(X,Y)=ag(X,Y)$ for any $X, Y\in H$, $a$ being a constant. If $\beta\neq \gamma$, $\beta\gamma\neq 0$, then $\alpha$, $a_1, h_1, \beta$ and $\gamma$ are constant.
\end{lemma}

\begin{proof}
First we prove that $\beta$ and $\gamma$ are constant. Using (\ref{9}), for any $e_x, e_y\in T_\beta$ and $e_s, e_t\in T_\gamma$,
\begin{eqnarray*}
& &e_y\beta = (\beta-\beta)g(\nabla_{e_x}e_y,e_x)=0,\\
& &e_s\gamma = (\gamma-\gamma)g(\nabla_{e_t}e_s,e_t)=0.
\end{eqnarray*}
Since $\beta\gamma=(2n+1)c-a$ is constant, we also have $e_x\gamma=0$ for any $e_x\in T_\beta$ and $e_s\beta=0$ for any $e_s\in T_\gamma$.

Next, by (\ref{10}), we obtain
\begin{equation*}
e_1\beta = (a_1-\beta)g(\nabla_{e_x}e_1, e_x)= (a_1-\beta)g(\nabla_{e_x}e_2, \phi e_x).
\end{equation*}
Since $M$ is non-Hopf, $h_1^2\neq 0$ locally, so $a_1\neq \beta$ and $a_1\neq \gamma$ on the neighborhood. Note that ${\rm{dim}}T_\beta\geq 2$ from Lemma 3.2 and Lemma 3.3. When $x\geq 3$, using $\phi e_x\in T_\gamma$, we have $g(\nabla_{e_x}e_2,\phi e_x)=0$ by (\ref{6}). So we obtain $e_1\beta=0$ and $g(\nabla_{e_x}e_1,e_x)=0$.

Moreover, (\ref{14}) induces
\begin{equation*}
\xi{\beta}=h_1g(\nabla_{e_x}e_1, e_x)=0.
\end{equation*}
From these equations, we see that $\beta$ and $\gamma$ are constant. 

Next we show that $a_1$ and $\alpha$ are constant. Since ${\rm{tr}}A=\beta+\gamma$ is constant, taking a trace of the shape operator $A$ yields
\begin{equation*}
a_1+\alpha= {\rm{tr}}A - p\beta - (2n-3-p)\gamma,
\end{equation*}
so $a_1 + \alpha$ is constant. We compute a sectional curvature for a plane spanned by $e_t\in T_\gamma$ and $e_1$. Using the equation of Gauss, we have
\begin{equation*}
g(R(e_t,e_1)e_1,e_t)=c+a_1\gamma.
\end{equation*}
On the other hand, we compute the right hand side of 
\begin{eqnarray*}
& &g(R(e_t,e_1)e_1,e_t)\\
& & =g(\nabla_{e_t}\nabla_{e_1}e_1,e_t) - g(\nabla_{e_1}\nabla_{e_t}e_1, e_t) - g(\nabla_{[e_t,e_1]}e_1,e_t).
\end{eqnarray*}
Since $a_1+\alpha$ is constant, (\ref{13}) and (\ref{17}) imply that
\begin{equation*}
(a_1-\gamma)g(\nabla_{e_1}e_1, e_t) + h_1 g(\nabla_{\xi}e_1,e_t)=0.
\end{equation*}
By (\ref{20}), we also have
\begin{equation*}
h_1 g(\nabla_{e_1}e_1, e_t) - (a_1-\gamma)g(\nabla_{\xi}e_1, e_t)=0.
\end{equation*}
These equations imply
\begin{equation*}
\{(a_1-\gamma)^2+h_1^2\}g(\nabla_{e_1}e_1, e_t)=0.
\end{equation*}
Since $h_1\neq 0$, we obtain $g(\nabla_{e_1}e_1,e_t)=0$ for any $e_t\in T_\gamma$. 

By (\ref{9}), we have $g(\nabla_{e_t}e_t, e_x)=0$ for any $e_x\in T_\beta$ and $e_t\in T_\gamma$. Thus we obtain
\begin{equation*}
g(\nabla_{e_t}\nabla_{e_1}e_1,e_t)=-g(\nabla_{e_1}e_1, \nabla_{e_t}e_t)=0.\
\end{equation*}
Next we compute the term $g(\nabla_{e_1}\nabla_{e_t}e_1, e_t)$. By (\ref{10}), we have $g(\nabla_{e_t}e_1, e_t)=0$, it follows that

\begin{eqnarray*}
& &g(\nabla_{e_1}\nabla_{e_t}e_1, e_t)\\
& &=-g(\nabla_{e_t}e_1, \nabla_{e_1}e_t)\\
& &= -g(\nabla_{e_t}e_1, \xi)g(\xi, \nabla_{e_1}e_t)-g(\nabla_{e_t}e_1, e_1)g(e_1,\nabla_{e_1}e_t)\\
& &\quad - \sum_{x}g(\nabla_{e_t}e_1, e_x)g(e_x, \nabla_{e_1}e_t) - \sum_{s} g(\nabla_{e_t}e_1, e_s)g(e_s, \nabla_{e_1}e_t).
\end{eqnarray*}
Taking a suitable orthonormal basis of $T_\beta$ and $T_\gamma$, when $x\neq 2$ and $g(\phi e_x, e_t)=0$, by (\ref{6}),
\begin{equation*}
g(\nabla_{e_t}e_1, e_x)=g(\nabla_{e_t}e_2, \phi e_x)=0,
\end{equation*}
since $e_2\in T_\beta$ and $\phi e_x\in T_\gamma$. We can take a suitable orthonormal basis such that $\phi e_x= e_t$. Then, by (\ref{9}), 
\begin{equation*}
g(\nabla_{e_t}e_2, e_t)=0=g(\nabla_{e_t}e_1, e_x).
\end{equation*}
When $x=2$, using (\ref{7}) and (\ref{8}), we see that
\begin{eqnarray*}
& &(\beta- a_1)g(\nabla_{e_t}e_2, e_1) - (\gamma- a_1) g(\nabla_{e_2}e_t, e_1)=0,\\
& &-h_1 g(\nabla_{e_t}e_2, e_1) + h_1 g(\nabla_{e_2}e_t, e_1)=0.
\end{eqnarray*}
Thus we have $g(\nabla_{e_t}e_2, e_1)=0$. Hence we obtain
\begin{equation}\label{26}
g(\nabla_{e_t}e_1, e_x)=0
\end{equation}
for any $e_x\in T_\beta$.
Moreover, by (\ref{11}), 
\begin{equation*}
(a_1-\gamma)g(\nabla_{e_t}e_1, e_s) + \gamma h_1 g(\phi e_t,e_s)=0.
\end{equation*}
Since $a_1\neq \gamma$, we see that 
$$g(\nabla_{e_t}e_1,e_s)=0.$$ 
From these equations, we obtain $g(\nabla_{e_1}\nabla_{e_t}e_1, e_t)=0$.

Next we compute 
\begin{eqnarray*}
& &g(\nabla_{[e_t,e_1]} e_1, e_t)\\
& &= g(\nabla_{\xi}e_1, e_t)(\xi, [e_t,e_1]) + g(\nabla_{e_1}e_1, e_t) g(e_1, [e_t,e_1])\\
& &\quad + \sum_{x}g(\nabla_{x} e_1, e_t) g(e_x, [e_t, e_1] ) + \sum_s g(\nabla_{e_s}e_1, e_t) g(e_s, [e_t,e_1]).
\end{eqnarray*}
Since $g(\phi e_1,e_t)=0$, we have $g(\xi, [e_t, e_1])=0$. By (\ref{8}) and (\ref{26}), it follows that
\begin{equation*}
h_1 g(\nabla_{e_x}e_1, e_t) = \{2c - 2\beta\gamma + \alpha(\beta + \gamma)\}g(\phi e_t, e_x).
\end{equation*}
By (\ref{26}), we have
\begin{equation*}
g(e_x, [e_t, e_1]) = g(e_x, \nabla_{e_t}e_1) - g(e_x, \nabla_{e_1}e_t) = -g(e_x, \nabla_{e_1}e_t).
\end{equation*}
Hence we obtain
\begin{eqnarray*}
g(\nabla_{[e_t,e_1]} e_1,e_t) &=& -\sum_{x} g(\nabla_{e_x}e_1, e_t) g(e_x, \nabla_{e_1}e_t)\\
 &=& -\frac{1}{h_1} \{2c-2\beta\gamma + \alpha(\beta+\gamma)\} g(\phi e_t,\nabla_{e_1}e_t).
\end{eqnarray*}
On the other hand,  (\ref{11}) and (\ref{26}) imply that
\begin{equation*}
g(\nabla_{e_1}e_t, \phi e_t)=\frac{-\gamma h_1}{\beta -\gamma}.
\end{equation*}
From these equations, we have
\begin{equation*}
g(\nabla_{[e_t,e_1]}e_1,e_t) = \frac{\gamma}{\beta-\gamma}\{2c - 2\beta\gamma +\alpha(\beta+\gamma)\}.
\end{equation*}
Summarizing the above we obtain
\begin{eqnarray*}
g(R(e_t,e_1)e_1,e_t)&=&-\frac{\gamma}{\beta-\gamma}\{2c-2\beta\gamma + \alpha(\beta+\gamma)\}\\
 &=&c+a_1\gamma,
\end{eqnarray*}
from which we see that
\begin{equation}\label{27}
(a_1 - \alpha)\gamma^2 = (\alpha + a_1)\beta\gamma +c(\beta + \gamma) - 2\beta\gamma^2.
\end{equation}
By the assumption, $\beta\gamma\neq 0$, and hence $\gamma\neq 0$.  Since $\beta, \gamma$ and $a_1+\alpha$ are constant,  (\ref{27}) implies that $a_1 - \alpha$ is a constant. So we see that  $a_1$ and $\alpha$ are constant. By (\ref{5}), $h_1$ is also constant.
\end{proof}

Next, we show that the case $\beta\gamma=0$ does not occur.

\begin{lemma}
Let $M$ be a non-Hopf real hypersurface of $M^n(c)$, $n\geq 3$, $c\neq 0$. Suppose that the Ricci tensor $S$ satisfies $S(X,Y)=ag(X,Y)$ for any $X, Y\in H$, $a$ being a constant. If  $\beta \neq \gamma$, then $\beta\gamma\neq 0$.
\end{lemma}

\begin{proof}
 We have $g(\phi e_x, e_y)=0$  for any $e_x,\ e_y\in T_\beta$ and $g(\phi e_s,e_t)=0$ for any $e_s,\ e_t\in T_{\gamma}$ by Lemma 3.2.

First, we suppose $\beta\neq 0$ and $\gamma=0$.

For any $e_s\in T_{\gamma}$, we consider
\begin{eqnarray*}
& &g(R(e_2,e_s)e_s,e_2)\\
& &=g(\nabla_{e_2}\nabla_{e_s}e_s,e_2)-g(\nabla_{e_s}\nabla_{e_2}e_s, e_2) - g(\nabla_{[e_2,e_s]}e_s, e_2).
\end{eqnarray*}
It follows from (\ref{9}) that $g(\nabla_{e_s}e_2,e_s)=0$. Thus we obtain
\begin{equation}\label{28}
g(\nabla_{e_2}\nabla_{e_s}e_s,e_2)=-g(\nabla_{e_s}e_s, \nabla_{e_2}e_2).
\end{equation}
We see that $g(\nabla_{e_s}e_s,\xi)=-g(e_s,\phi Ae_s)=0$. Moreover, (\ref{9}) and (\ref{10}) imply that $g(\nabla_{e_s}e_s,e_x)=0$ and $g(\nabla_{e_s}e_s,e_1)=0$. Therefore we get
\begin{eqnarray*}
g(\nabla_{e_s}e_s,\nabla_{e_2}e_2)=-\sum_t g(\nabla_{e_s}e_s,e_t)g(e_t,\nabla_{e_2}e_2), \hspace{0.5cm}  e_t\in T_{\gamma}.
\end{eqnarray*}
Since we have $g(\nabla_{e_2}e_2,e_t) = -g(\nabla_{e_2}e_1,\phi e_t)$, (\ref{11}) implies that 
\begin{equation}\label{29}
g(\nabla_{e_2}e_2,e_t)=0. 
\end{equation}
Consequently, by (\ref{28}), $g(\nabla_{e_2}\nabla_{e_s}e_s,e_2)=0$.

Next we compute
\begin{eqnarray*}
& &g(\nabla_{e_s}\nabla_{e_2}e_s,e_2)\\
& &=-g(\nabla_{e_2}e_s, \nabla_{e_s}e_2)\\
& &=-g(\nabla_{e_2}e_s,e_1)g(e_1,\nabla_{e_s}e_2).
\end{eqnarray*}
Using (\ref{7}) and (\ref{8}), we have
\begin{eqnarray*}
& &-a_1g(\nabla_{e_2}e_s, e_1) - (\beta-a_1)g(\nabla_{e_s}e_2, e_1)=0,\\
& &-h_1 g(\nabla_{e_2}e_s,e_1)+h_1g(\nabla_{e_s}e_2,e_1)=0.
\end{eqnarray*}
Thus we see that $g(\nabla_{e_2}e_s,e_1)=g(\nabla_{e_s}e_2,e_1)=0$. It follows that $g(\nabla_{e_s}\nabla_{e_2}e_s,e_2)=0$.

Since $\nabla_{e_2}e_t\in T_\gamma$ and $\nabla_{e_t}e_2\in T_\beta$ for any $e_t \in T_{\gamma}$, similar calculations can be performed to compute
\begin{equation*}
g(\nabla_{[e_2,e_s]} e_s,e_2)=0.
\end{equation*}
So we have $g(R(e_2,e_s)e_s,e_2)=0$ for $e_s \in T_{\gamma}$. On the other hand, the equation of Gauss implies
\begin{equation*}
g(R(e_2,e_s)e_s,e_2)=c.
\end{equation*}
This is a contradiction.

Next we assume $\beta=0$ and $\gamma\neq 0$. 

Now we suppose that $n>3$. Since $g(\phi e_x, e_y)=0$ for any $e_x, e_y \in T_{\beta}$ and $g(\phi e_s, e_t)=0$ for any $e_s, e_t \in T_{\gamma}$, we can take $e_x (\neq e_2)$ and $e_s$ such that $g(\phi e_x,e_s) =0$.

We compute $g(R(e_x,e_s)e_s,e_x)$. By (\ref{9}) and (\ref{10}), we have
\begin{eqnarray*}
& &g(\nabla_{e_x}e_x, e_1)=0,\\
& &g(\nabla_{e_s}e_s, e_1)=\frac{e_1\gamma}{\gamma-a_1},\\
& &g(\nabla_{e_x}e_x, e_s)=0,\\
& &g(\nabla_{e_s}e_s, e_x)=\frac{e_x\gamma}{\gamma}
\end{eqnarray*}
for any $e_x\in T_\beta$ and $e_s\in T_{\gamma}$. Since $g(\nabla_{e_x}e_x,\xi)=0$, we see that $\nabla_{e_x}e_x\in T_\beta$. Moreover, since $g(\nabla_{e_s}e_s, \xi)=0$, we can represent
\begin{equation}
\nabla_{e_s}e_s= \mu_1e_1 + \sum_{y}\mu_y e_y + \sum_t \mu_t e_t,
\end{equation}\label{30}
where we put $\mu_1=\frac{e_1\gamma}{\gamma-a_1}$,  $\mu_y=\frac{e_y\gamma}{\gamma}$ and $\mu_t = g(\nabla_{e_s}e_s,e_t)$. So we obtain

\begin{eqnarray*}
g(\nabla_{e_x}\nabla_{e_s}e_s,e_x)= (e_x\mu_x) + \sum_y \mu_y g(\nabla_{e_x}e_y, e_x).
\end{eqnarray*}

Next we compute 
\begin{eqnarray*}
g(\nabla_{e_s}\nabla_{e_x}e_s,e_x)=-g(\nabla_{e_x}e_s. \nabla_{e_s}e_x).
\end{eqnarray*}
By (\ref{6}), we see that $g(\nabla_{e_x}e_s,e_y)=0$ for any $x\neq y$ and $s$. We also have $g(\nabla_{e_s}e_x,e_t)=0$ for any $s\neq t$ and $x$. Since $\beta=0$, it follows that $g(\nabla_{e_x}e_s,\xi)=0$ for any $x$ and $s$. So we obtain
\begin{eqnarray*}
g(\nabla_{e_s}\nabla_{e_x}e_s,e_x)=-g(\nabla_{e_x}e_s,e_1)g(e_1,\nabla_{e_s}e_x).
\end{eqnarray*}
Here we can take $e_x\in T_\beta$ and $e_s\in T_\gamma$ such that $g(\phi e_x, e_s)=0$. Using (\ref{7}) and (\ref{8}), we have
\begin{eqnarray*}
& &(\gamma-a_1)g(\nabla_{e_x}e_s, e_1) + a_1g(\nabla_{e_s}e_x, e_1) = 0,\\
& &-h_1 g(\nabla_{e_x}e_s, e_1) + h_1 g(\nabla_{e_s}e_x, e_1)=0.
\end{eqnarray*}
From these, we see that $g(\nabla_{e_x}e_s, e_1)=g(\nabla_{e_s}e_x, e_1)=0$, and hence
$$g(\nabla_{e_s}\nabla_{e_x}e_s, e_x)=0,$$
where $g(\phi e_x, e_s)=0$.

Finally, we compute $g(\nabla_{[e_x,e_s]} e_s, e_x)$. Since we take $e_x, e_s$ such that $g(\phi e_x, e_s)=0$, it follows that $g(\nabla_{e_s}e_x, \xi)=0$ and $g(\nabla_{e_x}e_s,\xi)=0$. So we see that $\nabla_{e_x}e_s\in T_{\gamma}$ and $\nabla_{e_s}e_x \in T_\beta \oplus {\rm span}\{e_s\}$. 
\begin{eqnarray*}
g(\nabla_{[e_x, e_s]} e_s, e_x)=g(\nabla_{e_s}e_x,e_s)^2 =\mu_x^2.
\end{eqnarray*}
Since $g(\phi e_x, e_s)=0$, using $\phi e_s\in T_\beta$ and $\phi e_x\in T_\gamma$, 
\begin{equation*}
\mu_x=g(\nabla_{e_s}e_s, e_x)=g(\nabla_{e_s}\phi e_s, \phi e_x)=0.
\end{equation*}
Thus we have
$$g(\nabla_{[e_x, e_s]} e_s, e_x)=0.$$
These equations imply that
$$g(R(e_x,e_s)e_s,e_x)=\sum_y \mu_y g(\nabla_{e_x}e_y,e_x)$$
when $g(\phi e_x, e_s)=0$. Moreover, by (\ref{9}), for any $e_y\neq e_2$ and $e_t\in T_\gamma$, we obtain
\begin{eqnarray*}
0= -\gamma g(\nabla_{e_t}\phi e_y, \phi e_t) - (e_y\gamma).
\end{eqnarray*}
For each $e_y$,  we can take $e_t\in T_\gamma$ such that $g(\phi e_y, e_t)=0$. Then, by (\ref{6}),  we have $\mu_y=\frac{e_y\gamma}{\gamma}=0$ for any $y\neq 2$.
Therefore we obtain
\begin{eqnarray*}
g(R(e_x,e_s)e_s,e_x)=\mu_2 g(\nabla_{e_x}e_2, e_x).
\end{eqnarray*}

Next we compute $g(\nabla_{e_x}e_x,e_2)$. Since $\phi e_x\in T_\gamma$, using (\ref{7}) and (\ref{8}), we have
\begin{eqnarray*}
& &(\gamma-a_1)g(\nabla_{e_x}\phi e_x, e_1) + a_1g(\nabla_{\phi e_x}e_x, e_1) + h_1\gamma g(\phi e_x, \phi e_x)=0,\\
& &-h_1 g(\nabla_{e_x}\phi e_x, e_1)+ h_1 g(\nabla_{\phi e_x}e_x, e_1) + (2c+\alpha\gamma)g(\phi e_x, \phi e_x)=0.
\end{eqnarray*}
So we obtain
\begin{equation*}
g(\nabla_{e_x}e_x, e_2)=-g(\nabla_{e_x}\phi e_x ,e_1)=\frac{h_1^2\gamma - 2ca_1 - a_1\alpha\gamma}{h_1\gamma}.
\end{equation*} 
Similarly, we compute $g(\nabla_{e_s}e_s, e_2)=-g(\nabla_{e_s}\phi e_s,e_1)$ using (\ref{7}) and (\ref{8}). Then we have
\begin{equation*}
g(\nabla_{e_s}e_s,e_2)=-\frac{h_1^2\gamma + (2c+\alpha\gamma)(\gamma-a_1)}{h_1\gamma}.
\end{equation*}
These equations and the equation of Gauss imply that
\begin{eqnarray*}
& &g(R(e_x,e_s)e_s,e_x)\\
& &=\frac{(h_1^2\gamma - 2ca_1 -a_1\alpha\gamma)(h_1^2\gamma + (2c+\alpha\gamma)(\gamma-a_1))}{(h_1\gamma)^2}=-c.
\end{eqnarray*}
By the straightforward computation using ${\rm{tr}}A=\gamma=a_1+\alpha+q\gamma$, $q=(2n-3-p)$, we have
\begin{equation*}
-c=\frac{1}{\gamma^2}(q\gamma^2-2c)(2c-(q-1)\gamma^2),
\end{equation*}
from which
\begin{equation*}
q(q-1){\gamma^4} - (4q-1)c\gamma^2 + 4c^2=0.
\end{equation*}

So we see that $\gamma$ is constant, and hence, $\mu_2 = \frac{e_2\gamma}{\gamma}=0$. Therefore,  $g(R(e_x,e_s)e_s,e_x)=\mu_2 g(\nabla_{e_x}e_2,e_x)=0$. On the other hand, by the equation of Gauss, we have $g(R(e_x,e_s)e_s,e_x)=c$. This is a contradiction.

Finally, we consider the case that $n=3$.

We can take an orthonormal basis $\{\xi, e_1, e_2=\phi e_1, e_3, e_4=\phi e_3\}$, where $e_2, e_3 \in T_{\beta}$ and $e_4\in T_{\gamma}$, $\beta = 0$. Then we have
\begin{equation*}
\alpha + a_1 = 0, \hspace{1cm} h_1^2 = a_1(\gamma - a_1).
\end{equation*}
From these equations, we obtain
\begin{equation*}
2h_1 (e_2 h_1) =(e_2 \alpha)a_1 + \alpha(e_2 a_1) + (e_2 a_1)\gamma + a_1 (e_2\gamma) .
\end{equation*}
On the other hand, (\ref{12}), (\ref{13}), (\ref{16}) and (\ref{17}) imply that
\begin{eqnarray*}
& &e_2 h_1 = h_1 g(\nabla_{e_1}e_1, e_2) + (2c+\alpha a_1),\\
& &e_2 a_1 = a_1 g(\nabla_{e_1}e_1, e_2) + h_1 a_1,\\
& &e_2 h_1 = a_1g(\nabla_{\xi}e_1, e_2) + (c+ h_1^2),\\
& &e_2 \alpha = h_1 g(\nabla_{\xi}e_1, e_2) + h_1\alpha.
\end{eqnarray*}
Substituting these equations into the equation above, and using $h_1^2 - \alpha a_1 - a_1\gamma =0$ and $\alpha + a_1 = 0$, we obtain
\begin{equation*}
3ch_1 = a_1(e_2 \gamma).
\end{equation*}
By (\ref{9}), we have
\begin{equation*}
(e_2\gamma) = -\gamma g(\nabla_{e_4}e_2, e_4) = \gamma g(\nabla_{e_4}e_3, e_1).
\end{equation*}
Furthermore by (\ref{7}) and (\ref{8}), it follows that
\begin{eqnarray*}
& &(\gamma - a_1)g(\nabla_{e_3}e_4,e_1) + a_1 g(\nabla_{e_4}e_3,e_1) + \gamma h_1 =0,\\
& &-h_1 g(\nabla_{e_3}e_4,e_1) + h_1 g(\nabla_{e_4}e_3, e_1) + (2c+\alpha\gamma)=0,
\end{eqnarray*}
and hence
\begin{eqnarray*}
& &g(\nabla_{e_4}e_3, e_1) = \frac{2c(a_1-\gamma)}{\gamma h_1}.
\end{eqnarray*}
Thus we have
\begin{equation*}
(e_2\gamma) = \frac{2c(a_1 - \gamma)}{h_1}.
\end{equation*}
Therefore we have $3ch_1^2 = 2ca_1(a_1-\gamma)=-2ch_1^2$, which implies that $ch_1^2=0$. This is a contradiction.

From these considerations, we see that $\beta\gamma \neq 0$ for $n \geq 3$.
\end{proof}

From Lemma 3.4 and Lemma 3.5, we conclude that if a non-Hopf real hypersurface $M$ satisfies $S(X,Y)=ag(X,Y), X,Y\in H$ and if $\beta \neq \gamma$, then $\alpha, h_1, a_1, \beta $ and $\gamma$ are constant. Moreover, the principal curvatures of $M$ are constant.\\

\section[4]{Proof of Theorem}
 
 To prove our theorem we show that there does not exist a non-Hopf real hypersurface $M$ with the condition $S(X,Y) = ag(X,Y), X,Y \in H$, $a$ being a constant.

First we prove
\begin{lemma}
Let $M$ be a non-Hopf real hypersurface of $M^n(c)$, $n\geq 3$, $c\neq 0$. If the Ricci tensor $S$ satisfies $S(X,Y)=ag(X,Y)$ for any $X, Y\in H$, $a$ being a constant, then $\beta=\gamma$.
\end{lemma}

\begin{proof}
Suppose $\beta \neq \gamma$ at a point $x$ of $M$ and therefore in a neighborhood of $x$. We take a local orthonormal basis $\{\xi, e_1,\hdots,e_{2n-2}\}$ of a real hypersurface of $M$ of $M^{n}(c)$ as Section 3. By Lemma 3.4 and Lemma 3.5, we see that $\alpha, a_1, h_1, \beta$ and $\gamma$ are constant and $g(\phi e_x,e_y)=0 $ for any $e_x, e_y \in T_\beta$ and $g(\phi e_s,e_t)=0$ for any $e_s, e_t \in T_\gamma$.

By (\ref{16}) and (\ref{17}), we have
$$(c+\alpha \beta - \beta a_1 + h_1^2)-(a_1 - \beta)g(\nabla_{\xi}e_2,e_1)=0,$$
$$h_1(\alpha - 3\beta) - h_1 g(\nabla_{\xi}e_2, e_1)=0.$$
Since $h_1\neq 0$ and $h_1^2=(a_1 - \beta)(\gamma - a_1)$, we obtain
$$c+2\alpha\beta+3a_1\beta-3\beta^2+a_1\gamma-a_1^2-\beta\gamma-\alpha a_1=0.$$
On the other hand, by a straightforward computation shows that
\begin{eqnarray*}
& &g(R(e_1,\xi)\xi,e_1)=h_1 g(\nabla_{e_1}e_2,e_1)-(a_1 - \beta)g(\nabla_{e_\xi}e_2,e_1)+\beta a_1\\
& &=c+\alpha a_1 - h_1^2.
\end{eqnarray*}
By (\ref{13}) and (\ref{17}), we get
$$g(\nabla_{e_1}e_2,e_1)=\frac{h_1(2\beta + a_1)}{a_1 - \beta}, \hspace{1cm}g(\nabla_{\xi}e_2, e_1)=\alpha - 3\beta.$$
Substituting these equation into the equation above, and using  $h_1^2=(a_1 - \beta)(\gamma - a_1)$, we obtain
$$c+2\alpha a_1 - 2a_1\gamma + 2a_1^2-\beta\gamma-3a_1\beta-\alpha\beta+3\beta^2=0.$$
These equations imply that
$$c=\beta(\beta+\gamma-\alpha-a_1).$$
Now, we compute
$$g(R(e_t,e_2)e_2,e_t)=g(\nabla_{e_t}\nabla_{e_2}e_2,e_t)-g(\nabla_{e_2}\nabla_{e_t}e_2,e_t)-g(\nabla_{[e_t,e_2]}e_2,e_t)$$
for $e_t\in T_{\gamma}$. By (\ref{9}) and (\ref{10}), we have
$$g(\nabla_{e_t}\nabla_{e_2}e_2,e_t)=-g(\nabla_{e_2}e_2,e_1)g(e_1,\nabla_{e_t}e_t)=0.$$
Moreover, (\ref{6}) and (\ref{9}) imply
$$g(\nabla_{e_2}\nabla_{e_t}e_2,e_t)=-g(\nabla_{e_t}e_2,e_1)g(e_1,\nabla_{e_2}e_t).$$
Since $\beta\neq \gamma$, from (\ref{7}) and (\ref{8}), we have
$$g(\nabla_{e_2}e_t,e_1)=0,\ g(\nabla_{e_t}e_2,e_1)=0.$$ 
So we conclude
$$g(\nabla_{e_t}\nabla_{e_2}e_2,e_t)=0.$$
We also have, by (\ref{6}),
$$g(\nabla_{[e_t,e_2]}e_2,e_t)=0.$$
Consequently, we obtain
$$g(R(e_t,e_2)e_2,e_t)=c+\beta\gamma=0.$$
Therefore, we have $\alpha+a_1=\beta+2\gamma$. On the other hand, ${\rm tr}A=\beta+\gamma=\alpha+a_1+(n-1)\beta+(n-2)\gamma$, and hence $\beta+\gamma=0$. Substituting this into (\ref{27}), we get $(a_1-\alpha)\gamma^2=-\beta\gamma^2$. Since $\gamma\neq 0$, we have $\alpha-a_1=\beta$, which implies $2\alpha = 2\beta + 2\gamma= 0$ and $\alpha = 0$. Since $\gamma - a_1=\alpha = 0$, it follows that $h_1^2=(a_1-\beta)(\gamma-a_1)=0$, which is a contradiction. Therefore, we must have $\beta=\gamma$.

\end{proof}

\begin{lemma}
Let $M$ be a non-Hopf real hypersurface of $M^n(c)$ $(n\geq 3, c\neq 0)$. Suppose that the Ricci tensor $S$ satisfies $S(X,Y)=ag(X,Y)$ for any $X, Y\in T_0$, $a$ being a constant. If $\beta = \gamma$, then $\alpha, a_1, h_1 $ and $\beta$  are constant.
\end{lemma}

\begin{proof}
We suppose that $\beta= \gamma$.
We see that $\beta$ is a root of  a quadratic equation $X^2-({\rm{tr}}A) X+a-(2n+1)c=0$. We denote $\nu$ another root of this equation. We remark that $\beta\nu=a-(2n+1)c$ and $\beta+\nu={\rm tr}A$. By the equation
\begin{equation}\label{31}
g(Se_1,e_1)=(2n+1)c+ ({\rm{tr}}A) a_1 - (a_1^2 + h_1^2)=a,
\end{equation}
we have
\begin{equation}\label{32}
h_1^2= (a_1-\beta)({\rm{tr}}A - (a_1+\beta))=(a_1-\beta)(\nu-a_1).
\end{equation}
By (\ref{11}) and (\ref{15}), for any $e_i,\ e_j\in T_\beta$, we obtain
\begin{eqnarray*}
& &(a_1-\beta)g(\nabla_{e_i}e_1, e_j) + \beta h_1 g(\phi e_i,e_j)=0,\\
& &(c+\beta\alpha-\beta^2)g(\phi e_i,e_j) + h_1g(\nabla_{e_i}e_1, e_j)=0.
\end{eqnarray*}
These equations and (\ref{32}) imply
\begin{equation*}
(\beta(\nu-a_1) - (c+\beta\alpha -\beta^2))g(\phi e_i,e_j)=0.
\end{equation*}
There exist $e_i$ and $e_j$ that satisfy $g(\phi e_i, e_j)\neq 0$, so we have
\begin{equation}\label{33}
c=\beta(\beta+\nu-(a_1+\alpha)).
\end{equation}
Using ${\rm{tr}}A = a_1+\alpha+(2n-3)\beta$, we obtain
\begin{equation*}
c= \beta({\rm{tr}}A - a_1-\alpha) =(2n-3)\beta^2.
\end{equation*}
So we see that $\beta$ is constant. Since $\beta\nu=a-(2n+1)c$ is constant, $\nu$ is also constant.

In the following we put $\phi e_1=e_2$. 

When $j>2$, by the equation of Gauss,
\begin{eqnarray*}
& &g(R(e_j,e_1)e_1,e_j)\\
 & &=g(\nabla_{e_j}\nabla_{e_1}e_1, e_j) - g(\nabla_{e_1}\nabla_{e_j}e_1, e_j) - g(\nabla_{[e_j,e_1]}e_1,e_j)\\
  & &=c+a_1\beta.
\end{eqnarray*}
Since $a_1+\alpha$ is constant from (\ref{33}), by (\ref{13}) and (\ref{17}), we have
\begin{equation*}
0=e_j(\alpha + a_1)=h_1g(\nabla_{\xi}e_1, e_j) + (a_1-\beta)g(\nabla_{e_1}e_1,e_j).
\end{equation*}
On the other hand, by (\ref{20}),
\begin{equation*}
-(a_1-\beta)g(\nabla_{\xi}e_1, e_j) + h_1 g(\nabla_{e_1}e_1, e_j)=0.
\end{equation*}
From these equations, we obtain
\begin{equation*}
\{(a_1-\beta)^2 + h_1^2\} g(\nabla_{e_1}e_1, e_j)=0,
\end{equation*}
and hence 
\begin{equation}\label{34}
g(\nabla_{e_1}e_1, e_j)=0
\end{equation}
for any $j>2$. So we have
\begin{eqnarray*}
& &g(\nabla_{e_j}\nabla_{e_1}e_1, e_j)=-g(\nabla_{e_1}e_1, \nabla_{e_j}e_j)\\
 & &= -g(\nabla_{e_1}e_1, \xi)g(\xi, \nabla_{e_j}e_j) - g(\nabla_{e_1}e_1,e_1)g(e_1,\nabla_{e_j}e_j)\\
 & &\quad -g(\nabla_{e_1}e_1,e_2)g(e_2, \nabla_{e_j}e_j) - \sum_{k>2} g(\nabla_{e_1}e_1, e_k) g(e_k, \nabla_{e_j}e_j).
\end{eqnarray*}
From (\ref{34}) and $g(\nabla_{e_1}e_1, \xi)=0$, we obtain
\begin{equation*}
g(\nabla_{e_j}\nabla_{e_1}e_1, e_j) = -g(\nabla_{e_1}e_1, e_2)g(e_2, \nabla_{e_j}e_j).
\end{equation*}
Since $\phi e_1=e_2$, we have
\begin{equation*}
g(e_2, \nabla_{e_j}e_j)=g(\phi e_1, \nabla_{e_j}e_j)=-g(e_1, \nabla_{e_j}\phi e_j).
\end{equation*}
By (\ref{11}), 
\begin{equation*}
(a_1-\beta)g(\nabla_{e_j}e_1, \phi e_j) + \beta h_1 g(\phi e_j, \phi e_j)=0.
\end{equation*}
Thus we have
\begin{equation*}
g(\nabla_{e_j}e_j,e_2)=-\frac{\beta h_1}{a_1-\beta}.
\end{equation*}
From these equations, we obtain
\begin{equation*}
g(\nabla_{e_j}\nabla_{e_1}e_1, e_j)=\frac{\beta h_1}{a_1-\beta} g(\nabla_{e_1}e_1, e_2).
\end{equation*}
Next we compute $g(\nabla_{e_1}\nabla_{e_j}e_1, e_j)$. Since $\beta$ is constant and $h_1^2=(a_1-\beta)(\gamma-a_1)\neq 0$, by (\ref{10}), it follows that
\begin{equation}\label{35}
g(\nabla_{e_j}e_1, e_j)=0.
\end{equation}
Moreover, using (\ref{11}), 
\begin{eqnarray*}
(a_1-\beta)g(\nabla_{e_j}e_1, e_2) + \beta h_1 g(\phi e_j, e_2)=0,\\
(a_1-\beta)g(\nabla_{e_j}e_1, e_k) + \beta h_1 g(\phi e_j, e_k)=0
\end{eqnarray*}
for $k\geq 3$. So we have $g(\nabla_{e_j}e_1, e_2)=0$. Thus, taking a suitable orthonormal basis $\{e_3,\cdots,e_{2n-2}\}$, 
\begin{equation}\label{36}
g(\nabla_{e_j}e_1, \phi e_j)=\frac{-\beta h_1}{a_1-\beta}.
\end{equation}
So we have
\begin{eqnarray*}
& &g(\nabla_{e_1}\nabla_{e_j}e_1, e_j)\\
& &=-g(\nabla_{e_j}e_1, \nabla_{e_1}e_j)\\
& &=-g(\nabla_{e_j}e_1, \xi)g(\xi, \nabla_{e_1}e_j) - g(\nabla_{e_j} e_1, e_1) g(e_1, \nabla_{e_1}e_j)\\
& &\quad -g(\nabla_{e_j}e_1, e_2)g(e_2, \nabla_{e_1}e_j) - \sum_{k\geq 3} g(\nabla_{e_j}e_1, e_k) g(e_k, \nabla_{e_1}e_j)\\
& &=\frac{\beta h_1}{a_1-\beta} g(\phi e_j, \nabla_{e_1}e_j).
\end{eqnarray*}
Similar computation using (\ref{11}) induces
\begin{equation*}
g(\nabla_{e_k}e_1, e_j)=-\frac{\beta h_1}{a_1-\beta}g(\phi e_k,e_j).
\end{equation*}
On the other hand,  we have
\begin{eqnarray*}
& &g(\nabla_{[e_j, e_1]}e_1,e_j)\\
& &=g(\nabla_{\xi}e_1, e_j) g(\xi, [e_j, e_1]) + g(\nabla_{e_1}e_1, e_j)g(e_1, [e_j, e_1])\\
& &\quad \sum_{k>1} g(\nabla_{e_k}e_1, e_j)g(e_k, [e_j,e_1]).
\end{eqnarray*}
Since $g(\xi, [e_j, e_1])=0$ and $g(\nabla_{e_1}e_1,e_j)=0$ by (\ref{34}),
\begin{equation}\label{37}
g(\nabla_{[e_j,e_1]}e_1, e_j)=\frac{\beta h_1}{a_1-\beta} g(\phi e_j, [e_j,e_1]).
\end{equation}
From these, we obtain
\begin{eqnarray}\label{38}
g(R(e_j,e_1)e_1,e_j)&=&\frac{\beta h_1}{a_1-\beta}g(\nabla_{e_1}e_1, e_2) + \left( \frac{\beta h_1}{a_1-\beta} \right)^2 \\
&=&c+a_1\beta.\nonumber
\end{eqnarray}

Next we compute $g(\nabla_{e_1}e_1, e_2)$. Since ${\rm{tr}}A=\beta + \nu$ is constant, by (\ref{31}), we have
\begin{equation*}
2h_1(e_2h_1)=({\rm{tr}}A-2a_1)(e_2a_1).
\end{equation*}
Moreover, (\ref{12}) and (\ref{13}) imply that
\begin{eqnarray*}
& &(2c-2a_1\beta + \alpha(\beta+a_1))g(\phi e_2, e_1) + h_1 g(\nabla_{e_1}e_2, e_1) + (e_2h_1)=0,\\
& &h_1(2\beta+ a_1)g(\phi e_2, e_1) + (a_1-\beta)g(\nabla_{e_1}e_2, e_1) + (e_2a_1)=0.
\end{eqnarray*}
From these equations, we obtain
\begin{eqnarray*}
& &h_1(4c -4a_1\beta + 2\alpha(\beta + a_1) - (2\beta + a_1)({\rm{tr}}A -2a_1))g(\phi e_2, e_1)\\
& &\quad +(2h_1^2 - (a_1-\beta)({\rm{tr}}A-2a_1))g(\nabla_{e_1}e_2, e_1) =0.
\end{eqnarray*}
By a straightforward computation using (\ref{32}), (\ref{33})  and ${\rm{tr}}A=\beta+\nu$, we have
\begin{equation*}
g(\nabla_{e_1}e_2, e_1)=\frac{h_1(2\beta^2 + 2\beta\nu - 2\beta \alpha - 5a_1\beta + 2a_1\alpha -a_1\nu + 2a_1^2}{(a_1-\beta)(\nu-\beta)}.
\end{equation*}
By (\ref{38}), we obtain
\begin{eqnarray*}
& &g(R(e_j,e_1)e_1,e_j)\\
& &=\frac{\beta(\nu-a_1)(3\beta^2 + \beta\nu -2\beta\alpha - 5a_1\beta + 2a_1\alpha -a_1\nu +2a_1^2)}{
(a_1-\beta)(\beta-\nu)}\\
& &=c+\beta a_1.
\end{eqnarray*}
Hence we have
\begin{eqnarray*}
0&=&\beta\{4a_1\beta^2- \beta^3 - 3a_1\beta\alpha + \alpha\beta^2 - a_1\alpha\nu + \alpha\beta\nu\\
 & &-3\beta^2\nu + 6a_1\beta\nu - 3a_1^2\nu  -5a_1^2\beta +2a_1^2\alpha + 2a_1^3\}.
\end{eqnarray*}
We suppose $\beta\neq0$. Since $a_1+\alpha={\rm{tr}}A-(2n-3)\beta$, we put $\alpha=k-a_1$. We remark that $k$ is constant. Then  we have 
\begin{eqnarray*}
& &2a_1^2(-\beta-\nu +k) + a_1(3\beta^2-3\beta k -k\nu + 5\beta\nu)\\
& &\quad -\beta^3 + k\beta^2 + k\beta\nu -3\beta^2\nu =0.
\end{eqnarray*}
By  (\ref{33}), $\beta + \nu -k \neq 0$, hence $a_1$ is a root of the quadratic equation whose coefficients are all constant. So $a_1$, $\alpha$, $h_1$ are constant.

Finally we consider the case that $\beta=0$. Since $\beta^2-(\rm{tr}A)\beta+a-(2n+1)c=0$, we have $(2n+1)c=a$. Thus, from (\ref{31}), $0=({\rm tr}A)a_1 - (a_1^2 +h_1^2) = \alpha a_1 - h_1^2$. From this $a_1 \neq 0$. Moreover, (\ref{7}) implies $g(\nabla_{e_i}e_j, e_1) = g(\nabla_{e_j}e_i, e_1)$. Thus, by (\ref{8}), we have $c=0$. This is a contradiction.

\end{proof}

\begin{theorem} Let $M$ be a real hypersurface of a complex space form $M^n(c)$, $c\neq 0$, $n\geq 3$. The Ricci tensor $S$ of $M$ satisfies $S(X,Y)=ag(X,Y)$ for any $X,Y\in H$, $a$ being a constant, if and only if $M$ is a pseudo-Einstein real hypersurface.\end{theorem}

\begin{proof}

We suppose that $M$ is a non-Hopf real hypersurface  $M$ of $M^n(c)$, $n\geq 3$, $c\neq 0$ whose Ricci tensor $S$ satisfies $S(X,Y)=ag(X,Y)$ for any $X, Y\in H$, $a$ being a constant. Then Lemma 4.1 and Lemma 4.2 imply that $\beta = \gamma$ and  $\alpha, a_1, h_1$ and $\beta$ are constant. From (\ref{16}) and (\ref{17}), we have

\begin{eqnarray}
& &(c +\beta\alpha - \beta a_1 + h_1^2)g(\phi e_i,e_1) + (a_1-\beta)g(\nabla_{\xi}e_i,e_1)=0,\label{39}\\
& &h_1 (\alpha - 3\beta) g( e_1,\phi e_i) + h_1 g(\nabla_{\xi}e_i, e_1) =0.\label{40}
\end{eqnarray}
From these, using $h_1\neq 0$, we obtain

\begin{eqnarray}
c+\beta\alpha -\beta a_1 + h_1^2 - (\alpha - 3\beta)(a_1 - \beta) = 0.\label{41}
\end{eqnarray}
By (\ref{33}), we also have
$$c=(2n-3)\beta^2,$$
\begin{equation}
h_1^2=(a_1-\beta)({\rm tr}A - a_1 - \beta)=(a_1 - \beta)(\alpha + (2n-4)\beta).\label{42}
\end{equation}
Substituting these into the equation above, we have
$$\beta(2\beta - \alpha - (2n-2)a_1)=0.$$
If $\beta \neq 0$, then $h_1^2 = -(2n-2)(a_1-\beta)^2 \leq 0$. Thus we have $h_1 = 0$, which is a contradiction. So we see that $\beta = 0$. Then $h_1^2 - \alpha a_1= 0$ by (\ref{42}). On the other hand, from (\ref{41}),  $c+h_1^2 - \alpha a_1= c = 0$. This is also a contradiction. Therefore,  $M$ is a Hopf hypersurface. When $M$ is Hopf hypersurface that satisfies $S(X,Y)=ag(X,Y)$ for any $X, Y\in H$, then $M$ is pseudo-Einstein.

Conversely, if $M$ is pseudo-Einstein, then $S(X,Y)=ag(X,Y) + b\eta(X)\eta(Y)$, $a$ and $b$ being constant. Then $S$ satisfies $S(X,Y)=ag(X,Y)$ for any $X, Y\in H$.

\end{proof}

{\bf Remark.}
In Theorem 3.1 of \cite{MaK1}, we proved that if a real hypersurface $M$ of a complex space form $M^{n}(c)$, $c\ne 0$, with constant proncipal curvatures satisfies $S(X,Y)=ag(X,Y)$, $X, Y\in H$, $a$ being a function, then $M$ is a pseudo-Einstein real hypersurface.

If principal curvatures are constant, then the mean curvature vector field $ {\rm tr}A$ is also constant. Then we easily see that the function $a$ is constant. So, our result is an extension of the therem in  \cite{MaK1}.

\bibliographystyle{amsplain}

\end{document}